\documentclass[a4paper]{amsart}
\usepackage{color}
\usepackage[latin1]{inputenc}
\usepackage[T1]{fontenc}
\usepackage{amsfonts}
\usepackage{amssymb}
\usepackage{amsmath}
\usepackage{amsthm}
\usepackage{stmaryrd}
\usepackage{eufrak}
\usepackage{graphicx,type1cm,eso-pic,color}
\usepackage{pstricks,pst-plot,pstricks-add}
\usepackage{float}
\newtheorem{theorem}{Theorem}[section]
\newtheorem{lemma}[theorem]{Lemma}
\newtheorem{proposition}[theorem]{Proposition}
\newtheorem{corollary}[theorem]{Corollary}
\newtheorem{definition}[theorem]{Definition}
\newtheorem{assumption}[theorem]{Assumption}
\newtheorem{remark}[theorem]{Remark}
\newtheorem{example}[theorem]{Example}
\begin{document}
\setlength\arraycolsep{2pt}
\title{Matrix factorization for multivariate time series analysis}
\author{Pierre ALQUIER*}
\address{*RIKEN Center for Advanced Intelligence Project, Tokyo, Japan}
\email{pierre.alquier.stat@gmail.com}
\author{Nicolas MARIE$^{\dag}$}
\address{$^{\dag}$Laboratoire Modal'X, Universit\'e Paris Nanterre, Nanterre, France}
\email{nmarie@parisnanterre.fr}
\address{$^{\dag}$ESME Sudria, Ivry-sur-Seine, France}
\email{nicolas.marie@esme.fr}
\date{}
\maketitle
\noindent
%


%
\begin{abstract}
Matrix factorization is a powerful data analysis tool. It has been used in multivariate time series analysis, leading to the decomposition of the series in a small set of latent factors. However, little is known on the statistical performances of matrix factorization for time series. In this paper, we extend the results known for matrix estimation in the i.i.d setting to time series. Moreover, we prove that when the series exhibit some additional structure like periodicity or smoothness, it is possible to improve on the classical rates of convergence.
\end{abstract}
\noindent
\begin{small}\textbf{Acknowledgements.} The authors gratefully acknowledge Maxime Ossonce for discussions which helped to improve the paper. The first author was working at CREST, ENSAE Paris when this paper was written; he gratefully acknowledges financial support from Labex ECODEC (ANR-11-LABEX-0047).\end{small}
%


%

\section{Introduction}

Matrix factorization is a very powerful tool in statistics and data analysis. It was used as early as in the 70's in econometrics in reduced-rank regression~\cite{Izenman,Geweke,Kleibergen2}. There, matrix factorization is mainly a tool to estimate a coefficient matrix under a low-rank constraint. There was recently a renewed interest in matrix factorization as a data analysis tool for huge datasets. Nonnegative matrix factorization (NMF) was introduced by~\cite{LS1999} as a tool to represent a huge number of objects as linear combinations of elements of ``parts'' of objects. The method was indeed applied to large facial image datasets and the dictionary indeed contained typical parts of faces. Since then, various methods of matrix factorization were successfully applied such various fields as collaborative filtering and recommender systems on the Web~\cite{koren2009matrix,vernade2015learning}, document clustering~\cite{shahnaz2006document}, separation of sources in audio processing~\cite{ozerov2010multichannel}, missing data imputation~\cite{husson2018imputation}, quantum tomography~\cite{gross2011recovering,guctua2012rank,xia2016estimation,cai2016optimal,mai2017pseudo}, medical image processing \cite{AGT2014} topics extraction in texts \cite{paisley2015bayesian} or transports data analysis~\cite{carel2017non}. Very often, matrix factorization provides interpretable and accurate representations of the data matrix as the product of two much smaller matrices. The theoretical performances of matrix completion were studied in a series of papers by Cand\`es with many co-authors~\cite{candes2009exact,candes2010power,candes2010matrix}. Minimax rates for matrix completion and more general matrix estimation problems were derived in~\cite{koltchinskii2011nuclear,cai2015rop,klopp2017robust,klopp2017structured,moridomi2018tighter}. Bayesian estimators and aggregation procedures were studied in~\cite{alquier2013bayesian,suzuki2015convergence,alquier2015bayesian,alquier2017estimation,luu2017sharp,alquier2017oracle,dalalyan2018exponentially,lumbreras2018bayesian,dalalyan2018exponential,guedj2}.

To apply matrix factorization techniques to multivariate time series is a very natural idea. First, the low-rank structure induced by the factorization leads to high correlations that are indeed observed in some applications (this structure is actually at the core of cointegration models in econometrics~\cite{engle1987co,Kleibergen1,Bauwens}). Moreover, the factorization provides a decomposition of each series in a dictionary which member that can be interpreted as latent factors used for example in state-space models, see e.g. Chapter 3 in~\cite{koop2010bayesian}. For this reasons, matrix factorization was used in multivariate time series analysis beyond econometrics: electricity consumptions forecasting~\cite{de2017recovering,mei2018nonnegative}, failure detection in transports systems~\cite{tonnelier2018anomaly}, collaborative filtering~\cite{gultekin2017online}, social media analysis~\cite{saha2012learning} to name a few.

It is likely that the temporal structure in the data can be exploited to obtain an accurate and sensible factorization: autocorrelation, smoothness, periodicity... Indeed, while some authors use matrix factorization as a black box for data analysis, others propose in a way or another to adapt the algorithm to the temporal structure of the data~\cite{NIPS2016_6160,saha2012learning,cheung2015decomposing,gultekin2017online}. However, there is no theoretical guarantee that this leads to better predictions or better rates of convergence. Moreover, the aforementionned theoretical studies~\cite{candes2009exact,candes2010power,candes2010matrix,koltchinskii2011nuclear,cai2015rop,klopp2017robust,klopp2017structured,moridomi2018tighter} all assumed i.i.d noise, strongly limiting their applicability to study algorithms designed for time series such as in~\cite{NIPS2016_6160}. The objective of this paper is to address both issues.

Consider for example that one observes a $d$ series $(x_{i,t})_{t=1}^T=\mathbf{X}$ and assume that $\mathbf{X}=\mathbf{M}+\varepsilon$ where $\mathbf{M}$ is a rank $k$ matrix and $\varepsilon$ is some noise. In a first time, assume that entries $\varepsilon_{i,t}$ of $\varepsilon$ are i.i.d with variance $\sigma^2$. Theorem 3 in~\cite{koltchinskii2011nuclear} implies that there is an estimator $\hat{\mathbf{M}}_1$ of $\mathbf{M}$, such that $\frac{1}{dT}\|\hat{\mathbf{M}}_1 - \mathbf M\|_F^2 = \mathcal{O}(\sigma^2  \frac{k(d+T)}{dT}) $, up to log terms. Moreover, Theorem 5 in the same paper shows that this rate cannot be improved. Here, we propose an estimator $\hat{\mathbf{M}}=\hat{\mathbf{U}}\hat{\mathbf{V}}$, where $\hat{\mathbf{U}}$ is a $d\times k$ matrix and $\hat{\mathbf{V}}$ is $k\times T$. We study this estimator under the assumption that the rows $\varepsilon_{i,\cdot}$ of $\varepsilon$ are independent, centered, with covariance matrix $\Sigma_{\varepsilon}$, allowing a temporal dependence in the noise. We prove that $\frac{1}{dT}\|\hat{\mathbf{M}} - \mathbf M\|_F^2 = \mathcal{O}(\left\| \Sigma_{\varepsilon} \right\|_{{\rm op}} \frac{k(d+T)}{dT})$ where $\|\Sigma_\varepsilon \|_{{\rm op}}$ is the operator norm of $\Sigma_\varepsilon$. Note that in the i.i.d case $\Sigma_\varepsilon = \sigma^2 \mathbb{I}_T$, we recover the rate of~\cite{koltchinskii2011nuclear} as $\|\Sigma_\varepsilon \|_{{\rm op}} = \sigma^2$. However, our result is more general: we provide examples where the noise is non i.i.d and we still have a control on $\|\Sigma_\varepsilon \|_{{\rm op}}$. For example, when the noise is row-wise AR(1), that is $\varepsilon_{i,t+1} = \rho \varepsilon_{i,t} + \eta_{i,t}$ where the $\eta_{i,t}$ are i.i.d with variance $\sigma^2$ and $|\rho|<1$, we have $\|\Sigma_\varepsilon \|_{{\rm op}}\leqslant \sigma^2 \frac{1+|\rho|}{1-|\rho|} $.
 Moreover, our estimator can be tuned to take into account a possible periodicity or smoothness of the series. This is done by rewriting $\hat{\mathbf{W}}=\hat{\mathbf{V}}\mathbf{\Lambda}$ where $\mathbf{\Lambda}$ is a $\tau\times T$ matrix encoding the temporal structure, and $\tau \leqslant T$. In this case, we always improve on the rate $\mathcal{O}(\left\| \Sigma_{\varepsilon} \right\|_{{\rm op}} \frac{k(d+T)}{dT})$.
 
 We obtain the following rates, for some constant $C(\beta,L)$:

\bigskip

\begin{center}
\begin{tabular}{|c|c|c|c|}
 \hline
 & no structure & $\tau$-periodic case & $\beta$-smooth case
 \\
 \hline
 order of $\frac{1}{dT}\|\hat{\mathbf{M}} - \mathbf M\|_F^2$ & $ \frac{\|\Sigma_\varepsilon \|_{{\rm op}} k(d+T)}{dT}$ & $\frac{\|\Sigma_\varepsilon \|_{{\rm op}} k(d+\tau)}{dT}$ & $ 
 \frac{\|\Sigma_\varepsilon \|_{{\rm op}}kd}{dT} + \left(\frac{\|\Sigma_\varepsilon \|_{{\rm op}} k}{dT}\right)^{\frac{2\beta}{2\beta+1}} $
 \\
 \hline
\end{tabular}
\end{center}

\bigskip

All the results are first stated under a known structure, that is, we assume that we know the rank $k$, the period $\tau$ or the smoothness $\beta$ of the series. We provide at the end of the paper a model selection procedure that allows to obtain the same rates of convergence without assuming this prior knowledge.

Finally, we should mention the nice paper~\cite{richard2014link} where the authors studied time-evolving adjacency matrices for graphs with autoregressive features. However, the rows of an adjacency matrix are not interpreted as time series, so the objective of this work is quite different from ours.

The paper is organized as follows. In Section~\ref{section_notations} we introduce the notations that will be used in all the paper. In Subection~\ref{section_oracle_inequality_identity} we study matrix factorization without additional temporal structure. In Section~\ref{section_oracle_inequality_general_case}, we study the estimator $\hat{\mathbf{M}} = \hat{\mathbf{U}} \hat{\mathbf{V}} \mathbf{\Lambda}$ in the general case, and show how it improves the rates of convergence for a well chosen matrix $\mathbf{\Lambda}$ for periodic and/or smooth series. Finally, adaptation to unknown rank, periodicity and/or smoothness is tackled in Section~\ref{section_selection}. The proofs are given in Section~\ref{section_proofs}.
%


%
\section{Setting of the problem and notation}
\label{section_notations}

Assume that we observe a multivariate series
\begin{displaymath}
\mathbf X = (x_{i,t})_{(i,t)\in\llbracket 1,d\rrbracket\times\llbracket 1,T\rrbracket}.
\end{displaymath}
where $d\in\mathbb N^*$ and $T\in\mathbb N\backslash\{0,1\}$.
This multivariate series is modelled as a stochastic process. We actually assume that
\begin{equation}\label{model}
\mathbf X =\mathbf M +\varepsilon,
\end{equation}
where $\varepsilon$ is a noise and $\mathbf M$ is a matrix of rank $k\in\llbracket 1,T\rrbracket$. Then, there exist $\mathbf U\in\mathcal M_{d,k}(\mathbb R)$ and $\mathbf W\in\mathcal M_{k,T}(\mathbb R)$ such that $\mathbf M =\mathbf U\mathbf W$. We will refer to $\mathbf{W}$ as the {\it dictionary} or as the {\it latent series}.
\\
\\
We also want to model more structure in $\mathbf M$. This is done by rewritting $\mathbf{W}=\mathbf V \mathbf \Lambda$, where $\tau\in\mathbb{N}\setminus\{0\}$, $\mathbf V\in\mathcal M_{k,\tau}$ and $\mathbf \Lambda \in\mathcal M_{\tau,T}$, where $\mathbf\Lambda$ is a {\it known} matrix. The matrix $\Lambda$ depends on the structure assumed on $\mathbf{M}$.
%


%
\begin{example}[Periodic series]\label{periodic}
Assume that $T= p\tau$ with $p\in\mathbb N^*$ for the sake of simplicity. To assume that the latent series in the dictionary $\mathbf{W}$ are $\tau$-periodic is exactly equivalent to writing $
\mathbf W =\mathbf V \mathbf \Lambda
$
where $\mathbf V\in\mathcal{M}_{k,\tau}(\mathbb R)$ and
$
\mathbf\Lambda =
(\mathbf{I}_\tau |\dots |\mathbf{I}_\tau )
\in\mathcal{M}_{\tau,T}(\mathbb R)
$
is defined by blocks, $\mathbf{I}_\tau$ being the indentity matrix in $\mathcal{M}_{\tau,\tau}(\mathbb R)$.
\end{example}
%


%
\begin{example}[Smooth series]\label{smooth_function}
We can assume that the series in $\mathbf{W}$ are smooth. For example, say that they belong a a Sobolev space with smoothness $\beta$, we have
\begin{displaymath}
\mathbf W_{i,t} =
\sum_{n = 0}^{\infty}
\mathbf{U}_{i,n}\mathbf e_n\left(\frac{t}{T}\right)
\end{displaymath}
where $(\mathbf e_n)_{n\in\mathbb N}$ is the Fourier basis (the definition of a Sobolev space is reminded is Section~\ref{section_oracle_inequality_general_case} below). Of course, there are infinitely many coefficients $\mathbf{U}_{i,n}$ and to estimate them all is not feasible, however, for $\tau$ large enough, the approximation
\begin{displaymath}
\mathbf W_{i,t} \simeq
\sum_{n = 0}^{\tau-1}
\mathbf{U}_{i,n}\mathbf e_n\left(\frac{t}{T}\right)
\end{displaymath}
will be suitable, and can be rewritten as $\mathbf{W}=\mathbf{U}\mathbf{\Lambda}$ where $\mathbf\Lambda _{i,t}=e_i(t/T)$. More details will be given in Section~\ref{section_oracle_inequality_general_case}, where we actually cover more general basis of functions.
\end{example}
\noindent
So our complete model will finally be written as
\begin{equation}
\label{model1}
\mathbf X =\mathbf M +\varepsilon =
\mathbf U\mathbf V \mathbf\Lambda + \varepsilon,
\end{equation}
where $\mathbf U\in\mathcal M_{d,k}(\mathbb R)$ and $\mathbf V \in\mathcal M_{k,\tau}(\mathbb R)$ are unknown, but $\tau\leqslant T$ and $\mathbf\Lambda\in\mathcal M_{\tau,T}(\mathbb C)$ such that ${\rm rank}(\mathbf\Lambda) =\tau$ are known (note that the unstructured case corresponds to $\tau = T$ and $\mathbf\Lambda =\mathbf I_T$).
\\
\\
Note that more constraint can be imposed on the estimator. For example, in nonnegative matrix factorization~\cite{LS1999}, one imposes that all the entries in $\hat{\mathbf{U}}$ and $\hat{\mathbf{W}}$ are nonnegative. Here, we will more generally assume that $\hat{\mathbf{U}}\hat{\mathbf{V}}$ belongs to some prescribed subset $\mathcal{S}\subseteq \mathcal{M}_{d,T}(\mathbb{R}) $.
\\
\\
In what follows, we will consider two norms on $\mathcal{M}_{d,T}$. For a matrix $A$, the Frobenius norms is given by
\begin{displaymath}
\|\mathbf A\|_{F} =
\textrm{trace}(\mathbf A\mathbf A^*)^{1/2}.
\end{displaymath}
and the operator norm by
\begin{displaymath}
 \|\mathbf A\|_{\textrm{op}} =
 \sup_{\|x\|= 1}
 \|\mathbf Ax\|
 \end{displaymath}
where $\|\cdot\|$ is the Euclidean norm on $\mathbb{R}^T$.

\subsection{Estimation by empirical risk minimization}

By multiplying both sides in~\eqref{model1} by the pseudo-inverse $\mathbf\Lambda^+ =\mathbf\Lambda^*(\mathbf\Lambda\mathbf{\Lambda}^*)^{-1}$, we obtain the ``simplified model''
\begin{displaymath}
\label{simplified-model}
\widetilde{\mathbf X} =
\widetilde{\mathbf{M}} + \widetilde{\varepsilon}
\end{displaymath}
with $\widetilde{\mathbf X} = \mathbf X \mathbf \Lambda^+$, $\widetilde{\mathbf M} =\mathbf{UV}$ and $\widetilde{\varepsilon} = \varepsilon \mathbf \Lambda^+$. In this model, the estimation of $\widetilde{\mathbf{M}}$ can be done by empirical risk minimization:
\begin{equation}\label{estimator}
\widehat{\widetilde{\mathbf M}}_{\mathcal S}\in
\arg\min_{\mathbf A\in\mathcal S}\widetilde r(\mathbf A)
\end{equation}
where
\begin{displaymath}
\widetilde r(\mathbf A) =\|\mathbf A -\widetilde{\mathbf X}\|_{F}^{2}
\textrm{ $;$ }
\forall\mathbf A\in\mathcal M_{d,\tau}(\mathbb R).
\end{displaymath}
Therefore, we can define the estimator $\widehat{\mathbf M}_{\mathcal S} = \widehat{\widetilde{\mathbf M}}_{\mathcal S}\mathbf\Lambda$ of $\mathbf M$.
\\
\\
In Section~\ref{section_oracle_inequality}, we study the statistical performances of this estimator. The first step is done in Subsection~\ref{section_oracle_inequality_identity}, where we derive upper bounds on
\begin{displaymath}
\left\| \widehat{\widetilde{\mathbf M}}_{\mathcal S} -\widetilde{\mathbf M}\right\|_{F}^2.
\end{displaymath}
The corresponding upper bounds on
\begin{displaymath}
\left\| \widehat{\mathbf M}_{\mathcal S} -\mathbf M\right\|_{F}^2
\end{displaymath}
are derived in Subsection~\ref{section_oracle_inequality_general_case}.

%
\section{Oracle inequalities}
\label{section_oracle_inequality}

Throughout this section, assume that $\varepsilon$ fulfills the following..
\begin{assumption}\label{assumption_noise}
The rows of $\varepsilon$ are independent and have the same $T$-dimensional sub-Gaussian distribution, with second moment matrix $\Sigma_{\varepsilon}$. Moreover, $\varepsilon_{1,.}\Sigma_{\varepsilon}^{-1/2}$ is isotropic and has a finite sub-Gaussian norm
\begin{displaymath}
\mathfrak K_{\varepsilon} :=
\sup_{\|x\|=1}
\sup_{p\in [1,\infty[}
p^{-1/2}\mathbb E(|\langle  \varepsilon_{1,.}\Sigma_{\varepsilon}^{-1/2},x\rangle|^p)^{1/p}
<\infty.
\end{displaymath}
\end{assumption}
\noindent
In the sequel, we also consider $\widetilde{\mathfrak K}_{\varepsilon} :=\mathfrak K_{\varepsilon}^{2}\vee\mathfrak K_{\varepsilon}^{4}$.
\\
\\
We remind (see e.g Chapter 1 in~\cite{chafai2012interactions})  that when $X \sim \mathcal{N}(0,\mathbf I_n) $,
\begin{equation}
\label{sousgauss-gauss}
\sup_{\|x\|=1}
\sup_{p\in [1,\infty[}
p^{-1/2}\mathbb E(|\langle  X,x\rangle|^p)^{1/p}
=C
\end{equation}
for some universal constant $C>0$ (that is, $C$ does not depend on $n$). Thus, for Gaussian noise, Assumption~\ref{assumption_noise} is satisfied and $\mathfrak K_{\varepsilon} = C $ does not depend on the dimension $T$.


%
\subsection{The case $\mathbf\Lambda = \mathbf I_T$}\label{section_oracle_inequality_identity} In this subsection only, we assume that $\mathbf{\Lambda} = \mathbf I_T$ (and thus $\tau=T$). So the simplified model is actually the original model $
\widetilde{\mathbf{X}} = \mathbf{X}$, $\widetilde{\mathbf{M}} = \mathbf M$, $\widetilde{\mathbf{\varepsilon}} = \mathbf{\varepsilon}$ and
$\widehat{\widetilde{\mathbf M}}_{\mathcal S}=\widehat{\mathbf M}_{\mathcal S}$.

\begin{theorem}\label{oracle_inequality_main}
Under Assumption \ref{assumption_noise}, for every $\lambda\in ]0,1[$ and $s\in\mathbb R_+$,
\begin{displaymath}
\frac{1}{dT}
\left\| \widehat{\mathbf M}_{\mathcal S} - \mathbf{M} \right\|_F^2
\leqslant
\frac{1+\lambda}{1 -\lambda}\cdot
\min_{\mathbf A\in\mathcal S} \frac{1}{dT} \left\| \mathbf{A} - \mathbf{M} \right\|_F^2+
\frac{4\mathfrak c\widetilde{\mathfrak K}_{\varepsilon}\|\Sigma_{\varepsilon}\|_{\normalfont{\textrm{op}}}}{\lambda(1 -\lambda)}\cdot
\frac{k(d + T + s)}{dT}
\end{displaymath}
with probability larger than $1 - 2e^{-s}$.
\end{theorem}
\noindent
As a consequence, if we have indeed $\mathbf{M}\in\mathcal{S}$, then with large probability,
$$ \frac{1}{dT}\left\| \widehat{\mathbf M}_{\mathcal S} - \mathbf{M} \right\|_F^2
  \leqslant
\frac{4\mathfrak c\widetilde{\mathfrak K}_{\varepsilon}\|\Sigma_{\varepsilon}\|_{\normalfont{\textrm{op}}}}{\lambda(1 -\lambda)}
\cdot\frac{k(d + T + s)}{dT}.
$$
Thus, we recover the rate  $\mathcal{O}(\|\Sigma_\varepsilon\|_{{\rm op}} \frac{k(d+T)}{dT})$ claimed in the introduction.
\begin{remark}
 Since the bound relies on the constant $\|\Sigma_\varepsilon\|_{\normalfont{\textrm{op}}}$, let us provide its value in some special cases:
\begin{enumerate}
 \item If ${\rm cov}(\varepsilon_{1,t},\varepsilon_{1,t'}) = \sigma^2 \mathbf{1}_{\{t=t'\}}$ then
 $$
 \|\Sigma_{\varepsilon}\|_{\textrm{op}} = \sigma^2.
 $$
 More generally, when $\varepsilon_{1,1},\dots,\varepsilon_{1,T}$ are uncorrelated,
 \begin{displaymath}
 \|\Sigma_{\varepsilon}\|_{\textrm{op}} =\max_{t\in\llbracket 1,T\rrbracket} {\rm var}(\varepsilon_{1,t}).
 \end{displaymath}
 \item Let $(\eta_t)_{t\in\mathbb Z}$ be a white noise of standard deviation $\sigma > 0$ and assume that there exists $\theta\in\mathbb R^*$ such that $\varepsilon_{1,t} =\eta_t -\theta\eta_{t - 1}$ for every $t\in\llbracket 1,T\rrbracket$. In other words, $(\varepsilon_{1,t})_{t=1,\dots,T}$ is the restriction of a MA(1) process to $\llbracket 1,T\rrbracket$. So,
 \begin{displaymath}
 \Sigma_{\varepsilon} =
 \sigma^2
 \begin{pmatrix}
  1 +\theta^2 & -\theta & 0 & \dots & 0\\
  -\theta & 1 +\theta^2 & 0 & \dots & 0\\
  \vdots & & \ddots & & \vdots\\
  0 & \dots & 0 & 1 +\theta^2 & -\theta\\
  0 & \dots & 0 & -\theta & 1 +\theta^2
 \end{pmatrix}
 \end{displaymath}
 and then
 \begin{displaymath}
 \|\Sigma_{\varepsilon}\|_{\textrm{op}} =
 \sigma^2 \left[ 1 +\theta^2 - 2\theta\min_{\ell \in\llbracket 1,T\rrbracket}\cos\left(\frac{\ell \pi}{1 + T}\right) \right] \leqslant \sigma^2 (1+\theta)^2.
 \end{displaymath}
 \item Let $(\eta_t)_{t\in\mathbb Z}$ be a white noise of standard deviation $\sigma > 0$ and assume that there is a $\rho$ with $|\rho|<1$ such that $\varepsilon_{1,t} = \rho \varepsilon_{1,t-1} + \eta_t$. So $(\varepsilon_{1,t})_{t=1,\dots,T}$ is the restriction of a AR(1) process to $\llbracket 1,T\rrbracket$. So,
 \begin{displaymath}
 \Sigma_{\varepsilon} =
 \sigma^2
 \begin{pmatrix}
  1 & \rho & \rho^2 & \dots & \rho^{T-1} \\
  \rho & 1 & \rho & \dots & \rho^{T-2} \\
  \vdots & & \ddots & & \vdots\\
  \rho^{T-2} & \dots & \rho & 1 & \rho \\
  \rho^{T-1} & \dots & \rho^2 & \rho & 1
 \end{pmatrix}
 = \sigma^2\left[\mathbf I_T + \sum_{t=1}^{T-1} \rho^t \left( \mathbf{J}_T^t + (\mathbf{J}_T^*)^t \right) \right] .
 \end{displaymath}
 where
 $$
 \mathbf{J}_T =  \begin{pmatrix}
  0 & 1 & 0 & \dots & 0 \\
  0 & 0 & 1 & \dots & 0 \\
  \vdots & & \ddots & \ddots & \vdots\\
  0 & \dots & 0 & 0 & 1 \\
  0 & \dots & 0 & 0 & 0
 \end{pmatrix}.
 $$
 As $\|\mathbf{J}_T\|_{\rm op}=1$, we have
 $$ \| \Sigma_{\varepsilon} \|_{\rm op} \leqslant \sigma^2 \left(1+2\sum_{t=1}^T |\rho|^t \right) \leqslant \sigma^2 \left(1+\frac{2|\rho|}{1-|\rho|} \right) = \sigma^2 \frac{1+|\rho|}{1-|\rho|}. $$
\end{enumerate}
\end{remark}
%


%
\subsection{The general case}\label{section_oracle_inequality_general_case}
Let us now come back to the general case. An application of Theorem \ref{oracle_inequality_main} to the ``simplified model''~\eqref{simplified-model} shows that for any $\lambda\in ]0,1[$ and $s\in\mathbb R_+$,
\begin{equation}\label{oracle_general}
\left\|\widehat{\widetilde{\mathbf M}}_{\mathcal S} - \widetilde{\mathbf{M}} \right\|_ F^2
\leqslant
\frac{1+\lambda}{1 -\lambda}\cdot
\min_{\mathbf A\in\mathcal S} \left\| \mathbf{A \Lambda} - \widetilde{\mathbf{M}} \right\|_ F^2 +
\frac{4\mathfrak ck}{\lambda(1 -\lambda)}
(d +\tau + s)\widetilde{\mathfrak K}_{\widetilde\varepsilon}
\|\Sigma_{\varepsilon\mathbf\Lambda^+}\|_{\normalfont{\textrm{op}}}
\end{equation}
with probability larger than $1 - 2e^{-s}$.
\\
\\
In order to obtain the desired bound on $\|\widehat{\mathbf M}_{\mathcal S} - \mathbf{M}\|_ F^2$, we must now understand the behaviour of $\|\Sigma_{\varepsilon\mathbf\Lambda^+}\|_{\normalfont{\textrm{op}}}$ and $\widetilde{\mathfrak K}_{\widetilde\varepsilon}$.
%


%
\begin{lemma}\label{lemmaSigma}
For any matrix $\mathbf C\in\mathcal M_{T,\tau}(\mathbb C)$,
\begin{displaymath}
\|\Sigma_{\varepsilon_{1,.}\mathbf C}\|_{\normalfont{\textrm{op}}}
\leqslant
\|\Sigma_{\varepsilon}\|_{\normalfont{\textrm{op}}}
\|\mathbf C^*\mathbf C\|_{\normalfont{\textrm{op}}}.
\end{displaymath}
\end{lemma}
\noindent
The situation regarding $\widetilde{\mathfrak K}_{\widetilde\varepsilon} =\mathfrak K_{\widetilde\varepsilon}^{2}\vee\mathfrak K_{\widetilde\varepsilon}^{4}$ is different, we are not aware of a general simple upper bound on $\mathfrak K_{\tilde{\varepsilon}}=\mathfrak K_{\varepsilon \mathbf{\Lambda}^+}$ in terms of $\mathfrak K_{\varepsilon}$ and $ \mathbf{\Lambda}^+$. Still, there are two cases where we actually have $\mathfrak K_{\tilde{\varepsilon}}=\mathfrak K_{\varepsilon }$. Indeed, in the Gaussian case, $\mathfrak K_{\tilde{\varepsilon}}=\mathfrak K_{\varepsilon}=C$, see~\eqref{sousgauss-gauss} above. For non Gaussian noise, we have the following result.
%


%
\begin{lemma}\label{lemma_subGaussian_norm}
Assume that there is $c(\tau,T) > 0$ such that $\mathbf\Lambda\mathbf\Lambda^* = c(\tau,T)\mathbf I_{\tau}$. If $\Sigma_{\varepsilon} =\sigma^2\mathbf I_T$ with $\sigma > 0$, then $\mathfrak K_{\widetilde\varepsilon} =\mathfrak K_{\varepsilon}$.
\end{lemma}
\noindent
Note that the assumption on $\mathbf \Lambda$ is fullfilled by the examples covered in Subsections~\ref{periodic_time_series} and~\eqref{smooth_trend_time_series}.
\\
\\
The previous discussion legitimates the following assumption.
\begin{assumption}\label{assumption-mathfrakK}
 $\mathfrak K_{\widetilde\varepsilon} \leqslant \mathfrak K_{\varepsilon}$.
\end{assumption}
\noindent
Finally, note that
$$
\left\|\widehat{\mathbf M}_{\mathcal S} - \mathbf{M} \right\|_ F^2
 = \left\| (\widehat{ \widetilde{\mathbf M}}_{\mathcal S} - \widetilde{\mathbf{M}}) \mathbf{\Lambda} \right\|_ F^2
 \leqslant \left\| \widehat{ \widetilde{\mathbf M}}_{\mathcal S} - \widetilde{\mathbf{M}} \right\|_ F^2 \|\mathbf{\Lambda}\mathbf{\Lambda}^*\|_{\rm op}
$$
and in the same way
$$
\left\| \mathbf{A} - \widetilde{\mathbf{M}} \right\|_ F^2
 = \left\| (\mathbf{A}\mathbf{\Lambda} - \mathbf{M}) \mathbf{\Lambda}^+ \right\|_ F^2
 \leqslant \left\| \mathbf{A}\mathbf{\Lambda} - \mathbf{M} \right\|_ F^2 \|\mathbf{\Lambda}\mathbf{\Lambda}^+ \|(\mathbf{\Lambda}\mathbf{\Lambda}^*)^{-1} \|_{\rm op}.
$$
By Inequality~\eqref{oracle_general} together with Lemmas~\ref{lemmaSigma} and~\ref{lemma_subGaussian_norm}, we obtain the following result.
%


%
\begin{corollary}\label{oracle_inequality_corollary_fini}
Fix $\lambda\in ]0,1[$ and $s\in\mathbb R_+$. Under Assumption \ref{assumption_noise} and Assumption \ref{assumption-mathfrakK},
$$
\frac{1}{dT}
\left\|\widehat{\mathbf M}_{\mathcal S} - \mathbf{M} \right\|_ F^2
\leqslant 
\frac{1 +\lambda }{1 -\lambda}\cdot
\min_{\mathbf A\in\mathcal S}
\frac{1}{dT}
\left\| \mathbf{A}\mathbf{\Lambda} - \mathbf{M} \right\|_ F^2
+\frac{4\mathfrak c\widetilde{\mathfrak K}_{\varepsilon}\|\Sigma_{\varepsilon}\|_{\normalfont{\textrm{op}}}}{\lambda(1 -\lambda)}
\cdot\frac{k(d +\tau + s)}{dT}
$$
with probability larger than $1 - 2e^{-s}$.
\end{corollary}
\noindent
Corollary 3.7 provides an oracle inequality: it says that our estimator provides the optimal tradeoff between a variance term in $ \|\Sigma_{\varepsilon}\|_{\normalfont{\textrm{op}}} k(d+\tau)/(dT)$, and a bias term. The bias term is the distance of $\mathbf{M}$ to its best approximation by a matrix of the form $\mathbf{A \Lambda}$. In order to explicit the rates of convergence, assumptions can be made to upper-bound the bias term. We now apply Corollary~\ref{oracle_inequality_corollary_fini} in the case of periodic time series, and then in the case of smooth time series. In each case, we explicit the bias term and the rate of convergence.
\subsection{Application: periodic time series}\label{periodic_time_series}

In the case of $\tau$-periodic time series, remind that we assumed for simplicity that there is an integer $p$ such that $\tau p = T$ and we defined
\begin{displaymath}
\mathbf\Lambda =
(\mathbf{I}_\tau|\dots|\mathbf{I}_\tau)\in\mathcal{M}_{\tau,T}(\mathbb R).
\end{displaymath}
Then
\begin{displaymath}
\mathbf\Lambda\mathbf\Lambda^* =
\frac{T}{\tau}\mathbf I_{\tau} \Rightarrow   \|\mathbf{\Lambda}\mathbf{\Lambda}^*\|_{\rm op} = \frac{T}{\tau} \text{ and } \|(\mathbf\Lambda
\mathbf\Lambda^*)^{-1}\|_{\normalfont{\textrm{op}}} =
\frac{\tau}{T}.
\end{displaymath}
Therefore, by Corollary~\ref{oracle_inequality_corollary_fini}, for every $\lambda\in ]0,1[$ and $s\in\mathbb R_+$, under Assumptions~\ref{assumption_noise} and~\ref{assumption-mathfrakK},
\begin{displaymath}
 \frac{1}{dT}  \left\|\widehat{\mathbf M}_{\mathcal S} - \mathbf{M} \right\|_ F^2
\leqslant
\frac{1+\lambda }{1 -\lambda}\cdot
\min_{\mathbf A\in\mathcal S}  \frac{1}{dT} \left\|\mathbf A \mathbf{\Lambda} - \mathbf{M} \right\|_ F^2
+\frac{4\mathfrak c\widetilde{\mathfrak K}_{\varepsilon}\|\Sigma_{\varepsilon}\|_{\normalfont{\textrm{op}}}}{\lambda(1 -\lambda)}
\cdot
\frac{k(d +\tau + s)}{dT}
\end{displaymath}
with probability larger than $1 - 2e^{-s}$.
Now, define
$$ \mathcal{S} = \{\mathbf{A}\in\mathcal{M}_{n,T}(\mathbb{R})\text{: } {\rm rank}(\mathbf{A}) \leqslant k \text{ and } \forall i,\forall t, \mathbf{A}_{i,t+\tau} = A_{i,t} \} $$ and assume that $ \mathbf{M}\in\mathcal{S}$. Then,
$$
\frac{1}{dT}\left\|\widehat{\mathbf M}_{\mathcal S} - \mathbf{M} \right\|_ F^2=
\mathcal{O}\left(\|\Sigma_\varepsilon\|_{{\rm op}} \frac{k(d+\tau)}{dT}\right)
$$
which is indeed an improvement with respect to the rate obtained without taking the periodicity into account, that is $\mathcal{O}(\|\Sigma_\varepsilon\|_{{\rm op}}  \frac{k(d+T)}{dT})$.


%
\subsection{Application: time series with smooth trend}\label{smooth_trend_time_series}
Assume we are given a dictionary of functions $(\mathbf e_n)_{|n|\leqslant N}$ for some finite $N\in\mathbb N$. This dictionary can for example be a finite subset of a basis of an Hilbert space $(\mathbf e_n)_{n\in\mathbb Z}$, like the Fourier basis or a wavelet basis.
\\
\\
Define
\begin{displaymath}
\mathbf\Lambda_N =
\left(\mathbf e_n\left(\frac{t}{T}\right)\right)_{(n,t)\in\llbracket -N,N\rrbracket\times\llbracket 1,T\rrbracket}.
\end{displaymath}
Note that $\mathbf{\Lambda}_N$ is a $\tau\times T$ matrix where $\tau=2N+1$.
\\
\\
Assume that
\begin{displaymath}
\mathbf\Lambda_N\mathbf\Lambda_N^* =
T\mathbf I_{\tau}.
\end{displaymath}
This implies that $\|(\mathbf\Lambda_N
\mathbf\Lambda_N^*)^{-1}\|_{\normalfont{\textrm{op}}} = 1/T$ and $\|\mathbf\Lambda_N
\mathbf\Lambda_N^*\|_{\normalfont{\textrm{op}}} = T$.
This can be the case for a well-chosen basis, otherwise, we can apply the Gram-Schmidt to the dictionary of functions.
\begin{example}\label{discrete_orthogonal_Hilbert_basis}
(Fourier's basis) Consider the Fourier basis $(\mathbf e_n)_{n\in\mathbb Z}$ defined by
\begin{displaymath}
\mathbf e_n(x) =
e^{2i\pi nx}\textrm{ $;$ }
\forall n\in\mathbb Z\textrm{, }
\forall x\in\mathbb R.
\end{displaymath}
On the one hand, for every $n\in\llbracket -N,N\rrbracket$ and $t\in\llbracket 1,T\rrbracket$, $|e_n(t/T)| = 1$. On the other hand, for every $m,n\in\llbracket -N,N\rrbracket$ such that $m\not= n$,
\begin{eqnarray*}
 \sum_{t = 1}^{T}
 \mathbf e_n\left(\frac{t}{T}\right)\overline{\mathbf e_m\left(\frac{t}{T}\right)} & = &
 \sum_{t = 1}^{T}e^{2i\pi (n - m)t/T}\\
 & = &
 \frac{e^{2i\pi (n - m)/T}(1 - e^{2i\pi (n - m)})}{1 - e^{2i\pi (n - m)/T}} = 0.
\end{eqnarray*}
\end{example}
\noindent
Therefore, by Corollary \ref{oracle_inequality_corollary_fini}, for every $\lambda\in ]0,1[$ and $s\in\mathbb R_+$, under Assumptions~\ref{assumption_noise} and~\ref{assumption-mathfrakK},
\begin{equation}
\label{step-coro-smooth}
  \left\|\widehat{\mathbf M}_{\mathcal S} - \mathbf{M} \right\|_ F^2
\leqslant
\frac{1+\lambda }{1 -\lambda}\cdot
\min_{\mathbf A\in\mathcal S}\left\|\mathbf A \mathbf{\Lambda}_N - \mathbf{M} \right\|_ F^2
+\frac{4\mathfrak c\widetilde{\mathfrak K}_{\varepsilon}
\|\Sigma_{\varepsilon}\|_{\normalfont{\textrm{op}}}}{\lambda(1 -\lambda)}
\cdot
\frac{k(d +(2N+1) + s)}{dT}
\end{equation}
with probability larger than $1 - 2e^{-s}$. We will now show the consequences of these results when the rows of $\mathbf{M}$ are smooth in the sense that they belong to a given Sobolev ellipsoid. In this case, we will not have a $\mathbf{A}$ such that $\left\|\mathbf A \mathbf{\Lambda} - \mathbf{M} \right\|_ F^2 = 0$, but this quantity will be small and can be controlled as a function of $N$. We introduce a few definitions.
%


%
\begin{definition}\label{Sobolev_ellipsoid}
The Sobolev ellipsoid $W(\beta,L)$ is the set of functions $f : [0,1]\rightarrow\mathbb R$ such that $f$ is $\beta - 1$ times differentiable, $f^{(\beta-1)}$ is absolutely continuous and
\begin{displaymath}
\int_{0}^{1} f^{(\beta)}(x)dx
\leqslant L^2.
\end{displaymath}
\end{definition}
\noindent
From now, we assume that $\mathbf{e}_n(x)=e^{2 i \pi n x}$ is the Fourier basis. It is well-known from Chapter 1 in~\cite{tsybNP} that any $f\in W(\beta,L)$ and $x\in[0,1]$,
\begin{displaymath}
f(x) =
\sum_{n = -\infty}^{\infty} c_{n}(f) \mathbf e_n(x)
\end{displaymath}
and that there is a (known) constant $C(\beta,L) > 0$ such that
\begin{equation}
\label{approximation}
\frac{1}{T} \sum_{t=1}^T\left[ f\left(\frac{t}{T}\right) - \sum_{n\leqslant |N|} c_{n} \mathbf e_n\left(\frac{t}{T}\right) \right]^2
\leqslant C(\beta,L) N^{-2\beta}.
\end{equation}
%


%
\begin{definition}\label{matrix_factorization_Sobolev}
 We define $\mathcal{S}(k,\beta,L)\subset \mathcal{M}_{d,T}(\mathbb{R}) $ as the set of matrices $\mathbf{M}$ such that $\mathbf{M} = \mathbf{U} \mathbf{W} $, $\mathbf{U} \in \mathcal{M}_{k,T}(\mathbb{R})$, $\mathbf{W} \in \mathcal{M}_{d,k}(\mathbb{R})$ and
\begin{enumerate}
 \item For any $i\in \llbracket 1,d\rrbracket$, $\|\mathbf{U}_{i,\cdot} \|^2 \leqslant 1 $,
 \item For any $\ell \in \llbracket 1,k\rrbracket$ and $t\in \llbracket 1,T\rrbracket$, $\mathbf{W}_{\ell,t} = f_{\ell}\left(\frac{t}{T}\right) $ for some $f_\ell \in W(\beta,L)$.
\end{enumerate}
Denote $\mathbf{V}_{N,\mathbf{W}}=\left( c_n(f_{\ell})\right)_{\ell \leqslant k, |n|\leqslant N}$.
\end{definition}
\noindent
Then~\eqref{approximation} implies
\begin{equation*}
\frac{1}{dT} \|\mathbf{M}-\mathbf{U} \mathbf{V}_{N,\mathbf{W}} \mathbf{\Lambda}_N \|_F^2
\leqslant C(\beta,L) N^{-2\beta}.
\end{equation*}
Pluging this into~\eqref{step-coro-smooth} gives
$$
\frac{1}{dT}\left\|\widehat{\mathbf M}_{{\mathcal S}(k,\beta,L)} - \mathbf{M} \right\|_ F^2
\leqslant
\frac{1+\lambda }{1 -\lambda}\cdot
C(\beta,L) N^{-2\beta}
+\frac{4\mathfrak c\widetilde{\mathfrak K}_{\varepsilon}
\|\Sigma_{\varepsilon}\|_{\normalfont{\textrm{op}}}}{\lambda(1 -\lambda)}\cdot
\frac{k(d +\tau + s)}{dT}.
$$
If $\beta$ is known, an adequate optimization with respect to $N$ gives the following result.
%


%
\begin{corollary}\label{oracle_inequality_Sobolev}
\label{corofourier}
Assume that $\mathbf{M}\in\mathcal{S}(k,\beta,L)$. Under Assumptions~\ref{assumption_noise} and~\ref{assumption-mathfrakK}, the choice $N=\lfloor (dTC(\beta,L)/(\|\Sigma_\varepsilon\|_{{\rm op}}k))^{1/(2\beta+1)} \rfloor$ ensures
\begin{displaymath}
  \frac{1}{dT}\left\|\widehat{\mathbf M}_{{\mathcal S}(k,\beta,L)} - \mathbf{M} \right\|_ F^2
\leqslant
\mathcal{C} \left[
 \|\Sigma_\varepsilon\|_{{\rm op}} \frac{kd+ s}{dT} + C(\beta,L)^{\frac{1}{2\beta+1}} \left(\|\Sigma_\varepsilon\|_{{\rm op}} \frac{k}{dT}\right)^{\frac{2\beta}{2\beta+1}}\right]
\end{displaymath}
with probability larger than $1 - 2e^{-s}$, where $\mathcal{C}>0$ is some constant depending on $\lambda$, $\mathfrak c$ and $\mathfrak K_{\varepsilon}$.
\end{corollary}
\noindent
However, in practice, $\beta$ is not known - nor the rank $k$. This problem is tackled in the next section.
%


%
\section{Model selection}
\label{section_selection}

Assume that we have many possible matrices $\mathbf{\Lambda}_\tau$, for $\tau \in \mathcal{T} \subset \{1,\dots,T\} $ and for each $\tau$, many possible $\mathcal{S}_{\tau,k}$ for different possible ranks $k\in\mathcal K \subset\{1,\dots,d\wedge T\}$.
\\
\\
Consider $s\in\mathbb R_+$ and the penalized estimator $\widehat{\mathbf M}_s =\widehat{\mathbf M}_{\mathcal{S}_{\widehat\tau_s,\widehat k_s}}$ with
\begin{displaymath}
(\widehat\tau_s,\widehat k_s)\in
\arg\min_{(\tau,k)\in\mathcal T\times\mathcal K}
\left\{\left\| \widehat{\mathbf M}_{\mathcal{S}_{\tau,k}} - \mathbf X \right\|_F^2
+\textrm{pen}_{s+\tau+k} (\tau,k)\right\},
\end{displaymath}
where
\begin{displaymath}
\textrm{pen}_s(\tau,k) =
\frac{2\mathfrak ck}{\lambda}(d +\tau + s)\widetilde{\mathfrak K}_{\varepsilon}\|\Sigma_{\varepsilon}\|_{\normalfont{\textrm{op}}}.
\end{displaymath}
%


%
\begin{theorem}\label{model_selection}
Under Assumptions~\ref{assumption_noise} and~\ref{assumption-mathfrakK}, for every $\lambda\in ]0,1[$,
\begin{multline*}
\frac{1}{dT}\left\| \widehat{\mathbf M}_{\mathcal{S}_{\hat{\tau}_s,\hat{k}_s}} -\mathbf{M} \right\|_F^2
 \leqslant 
 \min_{
 \begin{tiny}
 \begin{array}{c}
 (\tau,k)\in\mathcal T\times\mathcal K
 \\
 \mathbf A\in\mathcal S_{\tau,k}
 \end{array}
 \end{tiny}
 }
 \Biggl\{\left(\frac{1+\lambda}{1-\lambda}\right)^2
  \frac{1}{dT}\left\| \mathbf{A}\mathbf{\Lambda}_\tau - \mathbf{M} \right\|_ F^2 
  \\
  +\frac{16\mathfrak c\widetilde{\mathfrak K}_{\varepsilon}\|\Sigma_{\varepsilon}\|_{\normalfont{\textrm{op}}}}{\lambda (1-\lambda)^2} \cdot\frac{k(d+\tau+s)}{dT}  \Biggr\}.
\end{multline*}
\newline
with probability larger than $1 - 2e^{-s}$.
\end{theorem}
\begin{remark}
The reader might feel uncomfortable with the fact that the model selection procedure leads to $\hat{k}_s$ and $\hat{\tau}_s$ that depend on the prescribed confidence level $s$. Note that if $k=k_0$ is known, that is $\mathcal{K}=\{k_0\}$, then it is clear from the definition that $\hat{\tau}_s$ actually does not depend on $s$.
\end{remark}
\noindent
As an application, assume that $\mathbf{M}\in \mathcal{S}(k,\beta,L)$ where $k$ is known, but $\beta$ is unknown. Then the model selection procedure is feasible as it does not depend on $\beta$, and it satisfies exactly the same rate as $\hat{\mathbf{M}}_{\mathcal{S}(k,\beta,L)}$ in Corollary~\ref{corofourier}.
%


%
\bibliographystyle{plain}
\bibliography{biblio}
%


%
\section{Proofs}
\label{section_proofs}
\subsection{Additional notations}
Let us first introduce a few additional notations.
\\
\\
First, for the sake of shortness, we introduce the estimation risk $R$ and the empirical risk $r$. These notations also make clear the fact that our estimator can be seen as an empirical risk minimizer.
\begin{displaymath}
R(\mathbf A) =
\|\mathbf A -\mathbf M\|_{F}^{2}
\text{ and }
r(\mathbf A) =
\|\mathbf A -\mathbf X\|_{F}^{2}
\textrm{ $;$ }
\forall\mathbf A\in\mathcal M_{d,T}(\mathbb R).
\end{displaymath}
Let $\Delta(\mathcal S) =\{\mathbf A -\mathbf B\textrm{ $;$ }\mathbf A,\mathbf B\in\mathcal S\}$.
\\
\\
For any $\mathbf A\in\mathcal M_{d,T}(\mathbb C)$, the spectral radius of $\mathbf A$ is given by
\begin{displaymath}
\rho(\mathbf A) :=
\max\{|\lambda|\textrm{ $;$ }\lambda\in\textrm{sp}(\mathbf A)\}.
\end{displaymath}
Note that $\|\mathbf A\|_{\textrm{op}}^{2} =\rho(\mathbf A\mathbf A^*) =\rho(\mathbf A^*\mathbf A)$.
\\
\\
For any subset $\mathcal K$ of $\mathcal M_{d,T}(\mathbb C)$,
\begin{displaymath}
\textrm{rk}(\mathcal K) =
\max\{\textrm{rank}(\mathbf A)\textrm{ $;$ }\mathbf A\in\mathcal K\}
\end{displaymath}
and
\begin{displaymath}
\mathcal K^1 =\{\mathbf A\in\mathcal K :\|\mathbf A\|_{F}\leqslant 1\}.
\end{displaymath}
%


%
\subsection{Some lemmas}
Let us now state the key lemmas for the proof of our results. The first one will be used to estimate how far from the minimizer of $R$ is the minimizer of $r$.
%


%
\begin{lemma}\label{preliminary_estimates}
For any $\mathbf A\in\mathcal M_{d,T}(\mathbb R)$,
\begin{displaymath}
R(\mathbf A) - r(\mathbf A) +\|\varepsilon\|_{F}^{2} =
2\langle\varepsilon,\mathbf A -\mathbf M\rangle_{F}.
\end{displaymath}
Moreover, for every $\lambda\in ]0,1[$,
\begin{equation}\label{preliminary_estimates_1}
R(\mathbf A)
\leqslant
\frac{r(\mathbf A) -\|\varepsilon\|_{F}^{2}}{1 -\lambda} +
\frac{1}{\lambda(1 -\lambda)}
\left\langle\varepsilon,\frac{\mathbf A -\mathbf M}{\|\mathbf A -\mathbf M\|_{F}}\right\rangle_{F}^{2}
\end{equation}
and
\begin{equation}\label{preliminary_estimates_2}
r(\mathbf A) -\|\varepsilon\|_{F}^{2}
\leqslant
(1 +\lambda)R(\mathbf A) +
\frac{1}{\lambda}
\left\langle\varepsilon,\frac{\mathbf A -\mathbf M}{\|\mathbf A -\mathbf M\|_{F}}\right\rangle_{F}^{2}.
\end{equation}
\end{lemma}
\begin{proof}
Consider $\mathbf A\in\mathcal M_{d,T}(\mathbb R)$. First of all,
\begin{eqnarray*}
 R(\mathbf A) - r(\mathbf A) & = &
 \|\mathbf A -\mathbf X\|_{F}^{2} -\|\mathbf A -\mathbf M\|_{F}^{2}\\
 & = &
 \langle\mathbf X -\mathbf M,2\mathbf A -\mathbf X -\mathbf M\rangle_{F}^{2}\\
 & = &
 -\|\varepsilon\|_{F}^{2} + 2\langle\varepsilon,\mathbf A -\mathbf M\rangle_{F}.
\end{eqnarray*}
Then, for any $\lambda\in ]0,1[$,
\begin{equation}\label{preliminary_estimates_3}
R(\mathbf A) - r(\mathbf A) +
\|\varepsilon\|_{F}^{2} =
2\sqrt{\lambda R(\mathbf A)}
\left\langle\varepsilon,
\frac{\mathbf A -\mathbf M}{\sqrt{\lambda}
\cdot\|\mathbf A -\mathbf M\|_{F}}
\right\rangle_{F}.
\end{equation}
On the one hand, by Equation (\ref{preliminary_estimates_3}) together with the classic inequality $2ab\leqslant a^2 + b^2$ for every $a,b\in\mathbb R$,
\begin{displaymath}
R(\mathbf A) - r(\mathbf A) +
\|\varepsilon\|_{F}^{2}
\leqslant
\lambda R(\mathbf A) +
\frac{1}{\lambda}
\left\langle\varepsilon,
\frac{\mathbf A -\mathbf M}{\|\mathbf A -\mathbf M\|_{F}}
\right\rangle_{F}^{2}.
\end{displaymath}
So, Inequality (\ref{preliminary_estimates_1}) it true.
\\
\\
On the other hand, by Equation (\ref{preliminary_estimates_3}) together with the classic inequality $-2ab\leqslant a^2 + b^2$ for every $a,b\in\mathbb R$,
\begin{displaymath}
r(\mathbf A) - R(\mathbf A) -
\|\varepsilon\|_{F}^{2}
\leqslant
\lambda R(\mathbf A) +
\frac{1}{\lambda}
\left\langle\varepsilon,
\frac{\mathbf A -\mathbf M}{\|\mathbf A -\mathbf M\|_{F}}
\right\rangle_{F}^{2}.
\end{displaymath}
So, Inequality (\ref{preliminary_estimates_2}) it true.
\end{proof}
\noindent
In the proof of the theorems, $\mathbf{A}$ will be replaced by an estimator of $\mathbf{M}$ that will be data dependent. Thus, it is now crucial to obtain uniform bounds on the scalar product in Lemma~\ref{preliminary_estimates}. In machine learning theory, concentration inequalities are the standard tools to derive such a uniform bound, see~\cite{boucheron2013concentration} for a comprehensive introduction to concentration inequalities for independent observations, and their applications to statistics. Some inequalities for time series can be found for example in~\cite{dedecker2007weak}, and were applied to machine learning in~\cite{alquier2011sparsity}. Here, we require more specifically a concentration inequality on random matrices. Such inequalities can be found in~ \cite{tropp2012user,vershynin2010introduction}. We will actually use the following result (Theorem 5.39 and Remark 5.40.2 from~\cite{vershynin2010introduction}). As the proof can be found in~\cite{vershynin2010introduction}, we don't reproduce it here.
\begin{proposition}\label{concentration_inequality_lemma}
Under Assumption \ref{assumption_noise}, there exists a deterministic constant $\mathfrak m > 1$, not depending on $\varepsilon$, $d$ and $T$, such that for every $s\in\mathbb R_+$,
\begin{displaymath}
\left\|\frac{1}{d}\varepsilon^*\varepsilon -\Sigma_{\varepsilon}\right\|_{\normalfont{\textrm{op}}}
\leqslant
\mathfrak m
\max\left\{
\sqrt{\frac{T}{d}} +\sqrt{\frac{s}{d}};\left(\sqrt{\frac{T}{d}} +\sqrt{\frac{s}{d}}\right)^2
\right\}
\widetilde{\mathfrak K}_{\varepsilon}
\|\Sigma_{\varepsilon}\|_{\normalfont{\textrm{op}}}
\end{displaymath}
with probability larger than $1 - 2e^{-s}$, where $\widetilde{\mathfrak K}_{\varepsilon} :=\mathfrak K_{\varepsilon}^{2}\vee\mathfrak K_{\varepsilon}^{4}$.
\end{proposition}
\noindent
We are now in position to provide a uniform bound on the scalar product in Lemma~\ref{preliminary_estimates}.
\begin{lemma}\label{concentration_inequality}
Under Assumption \ref{assumption_noise}, there exists a constant $\mathfrak c > 1$, not depending on $\varepsilon$, $d$ and $T$, such that for every $s\in\mathbb R_+$ and $\mathcal K\subset\mathcal M_{d,T}(\mathbb R)$,
\begin{displaymath}
\sup_{\mathbf A\in\mathcal K^1}
\langle\varepsilon,\mathbf A\rangle_{F}^{2}
\leqslant
\mathfrak c\cdot\normalfont{\textrm{rk}}(\mathcal K^1)
(d + T + s)
\widetilde{\mathfrak K}_{\varepsilon}
\|\Sigma_\varepsilon\|_{\normalfont{\textrm{op}}}
\end{displaymath}
with probability larger than $1 - 2e^{-s}$.
\end{lemma}

\begin{proof}
Consider a subset $\mathcal K$ of $\mathcal M_{d,T}(\mathbb R)$ and $s\in\mathbb R_+$. Let $\sigma_1(\varepsilon)\geqslant\dots\geqslant\sigma_d(\varepsilon)$ be the singular values of $\varepsilon$.
On the one hand, consider a matrix $\mathbf A\in\mathcal K^1$ with singular values $\sigma_1(\mathbf A)\geqslant\dots\geqslant\sigma_{\textrm{rk}(\mathcal K^1)}(\mathbf A)$. By Cauchy-Schwarz's inequality:
\begin{eqnarray*}
 |\langle\varepsilon,\mathbf A\rangle_{F}|
 & \leqslant &
 \sum_{i = 1}^{\textrm{rk}(\mathcal K^1)}\sigma_i(\varepsilon)\sigma_i(\mathbf A)\\
 & \leqslant &
 \left|\sum_{i = 1}^{\textrm{rk}(\mathcal K^1)}\sigma_i(\varepsilon)^2\right|^{1/2}
 \|\mathbf A\|_{F}
 \leqslant
 \textrm{rk}(\mathcal K^1)^{1/2}\sigma_1(\varepsilon).
\end{eqnarray*}
Then,
\begin{equation}\label{concentration_inequality_1}
\sup_{\mathbf A\in\mathcal K^1}
\langle\varepsilon,\mathbf A\rangle_{F}^{2}
\leqslant
\textrm{rk}(\mathcal K^1)\sigma_1(\varepsilon)^2.
\end{equation}
On the other hand, consider
\begin{displaymath}
\omega\in\left\{
\left\|\frac{1}{d}\varepsilon^*\varepsilon -\Sigma_{\varepsilon}\right\|_{\normalfont{\textrm{op}}}
\leqslant
\mathfrak m
\max\left\{
\sqrt{\frac{T}{d}} +\sqrt{\frac{s}{d}};\left(\sqrt{\frac{T}{d}} +\sqrt{\frac{s}{d}}\right)^2
\right\}\widetilde{\mathfrak K}_{\varepsilon}\|\Sigma_{\varepsilon}\|_{\normalfont{\textrm{op}}}\right\}.
\end{displaymath}
Then,
\begin{eqnarray*}
 \left|\frac{1}{d}\sigma_1(\varepsilon(\omega))^2 -\|\Sigma_{\varepsilon}\|_{\textrm{op}}\right|
 & = &
 \left|\frac{1}{d}\|\varepsilon(\omega)\|_{\textrm{op}}^{2} -\|\Sigma_{\varepsilon}\|_{\textrm{op}}\right|\\
 & = &
 \left|\frac{1}{d}\|\varepsilon(\omega)^*\varepsilon(\omega)\|_{\textrm{op}} -\|\Sigma_{\varepsilon}\|_{\textrm{op}}\right|\\
 & \leqslant &
 \left\|\frac{1}{d}
 \varepsilon(\omega)^*\varepsilon(\omega)
 -\Sigma_{\varepsilon}\right\|_{\textrm{op}}\\
 & \leqslant &
 \mathfrak m
 \left(\sqrt{\frac{T}{d}} +\sqrt{\frac{s}{d}} +
 \frac{2T}{d} +\frac{2s}{d}\right)\widetilde{\mathfrak K}_{\varepsilon}
 \|\Sigma_{\varepsilon}\|_{\normalfont{\textrm{op}}}.
\end{eqnarray*}
In particular,
\begin{eqnarray}
 \sigma_1(\varepsilon(\omega))^2
 & \leqslant &
 \mathfrak m(\sqrt{Td} +\sqrt{sd} + 2T + 2s + d)\widetilde{\mathfrak K}_{\varepsilon}
 \|\Sigma_{\varepsilon}\|_{\textrm{op}}
 \nonumber\\
 & \leqslant &
 \mathfrak m\left(2d +\frac{5}{2}T +\frac{5}{2}s\right)
 \widetilde{\mathfrak K}_{\varepsilon}
 \|\Sigma_{\varepsilon}\|_{\textrm{op}}
 \nonumber\\
 \label{concentration_inequality_2}
 & \leqslant &
 \mathfrak c
 \left(d + T + s\right)
 \widetilde{\mathfrak K}_{\varepsilon}\|\Sigma_{\varepsilon}\|_{\normalfont{\textrm{op}}}
\end{eqnarray}
with $\mathfrak c = 5\mathfrak m/2$. Therefore, by Inequalities (\ref{concentration_inequality_1}) and (\ref{concentration_inequality_2}) together with Proposition \ref{concentration_inequality_lemma},
\begin{displaymath}
\sup_{\mathbf A\in\mathcal K^1}
\langle\varepsilon,\mathbf A\rangle_{F}^{2}
\leqslant
\mathfrak c\cdot\textrm{rk}(\mathcal K^1)
(d + T + s)
\widetilde{\mathfrak K}_{\varepsilon}
\|\Sigma_{\varepsilon}\|_{\normalfont{\textrm{op}}}
\end{displaymath}
with probability larger than $1 - 2e^{-s}$.
\end{proof}
\subsection{Proof of Theorem \ref{oracle_inequality_main}}
Consider $\lambda\in ]0,1[$ and $s\in\mathbb R_+$. By applying the Inequalities (\ref{preliminary_estimates_1}) and (\ref{preliminary_estimates_2}) of Lemma \ref{preliminary_estimates} successively:
\begin{eqnarray*}
 R(\widehat{\mathbf M}_{\mathcal S})
 & \leqslant &
 \frac{r(\widehat{\mathbf M}_{\mathcal S}) -\|\varepsilon\|_{F}^{2}}{1 -\lambda} +
 \frac{1}{\lambda(1 -\lambda)}
 \left\langle\varepsilon,\frac{\widehat{\mathbf M}_{\mathcal S} -\mathbf M}{\|\widehat{\mathbf M}_{\mathcal S}
 -\mathbf M\|_{F}} \right\rangle_{F}^{2}\\
 & \leqslant &
 \frac{1 +\lambda}{1 -\lambda}\cdot
 \min_{\mathbf A\in\mathcal S}R(\mathbf A) +
 \frac{2}{\lambda(1 -\lambda)}
 \sup_{\mathbf A\in\Delta(\mathcal S)^1}
 \langle\varepsilon,\mathbf A\rangle_{F}^{2}.
\end{eqnarray*}
By Lemma \ref{concentration_inequality}:
\begin{displaymath}
R(\widehat{\mathbf M}_{\mathcal S})
\leqslant
\frac{1 +\lambda}{1 -\lambda}\cdot
\min_{\mathbf A\in\mathcal S}R(\mathbf A) +
\frac{4\mathfrak ck}{\lambda(1 -\lambda)}
(d + T + s)
\widetilde{\mathfrak K}_{\varepsilon}\|\Sigma_{\varepsilon}\|_{\textrm{op}}
\end{displaymath}
with probability larger than $1 - 2e^{-s}$.
%


%
\subsection{Proof of Lemma \ref{lemmaSigma}}
First of all,
\begin{displaymath}
\Sigma_{\varepsilon_{1,.}\mathbf C}
=\mathbb E(\mathbf C^*\varepsilon_{1,.}^{*}\varepsilon_{1,.}\mathbf C)
=\mathbf C^*\mathbb E(\varepsilon_{1,.}^{*}\varepsilon_{1,.})\mathbf C
=\mathbf C^*\Sigma_{\varepsilon}\mathbf C.
\end{displaymath}
Then, since the matrix $\Sigma_{\varepsilon\mathbf C}$ is Hermitian,
\begin{eqnarray*}
 \|\Sigma_{\varepsilon_{1,.}\mathbf C}\|_{\normalfont{\textrm{op}}}
 & = &
 \sup_{x\in\mathbb C^{\tau}\backslash\{0\}}
 \frac{\|\mathbf C^*\Sigma_{\varepsilon}\mathbf Cx\|}{\|x\|} =
 \sup_{x\in\mathbb C^{\tau}\backslash\{0\}}
 \frac{x^*\mathbf C^*\Sigma_{\varepsilon}\mathbf Cx}{\|x\|^2}\\
 & = &
 \sup_{x\in\mathbb C^{\tau}\backslash\{0\}}
 \frac{x^*\mathbf C^*
 \Sigma_{\varepsilon}
 \mathbf Cx}{\|\mathbf Cx\|^2}\times
 \frac{\|\mathbf Cx\|^2}{\|x\|^2}\\
 & \leqslant &
 \left(\sup_{y\in\mathbb C^T\backslash\{0\}}
 \frac{y^*\Sigma_\varepsilon y}{\|y\|^2}\right)\left(
 \sup_{x\in\mathbb C^{\tau}\backslash\{0\}}
 \frac{\|\mathbf Cx\|^2}{\|x\|^2}\right)
 =\|\Sigma_\varepsilon\|_{\normalfont{\textrm{op}}}
 \|\mathbf C^*\mathbf C\|_{\normalfont{\textrm{op}}}.
\end{eqnarray*}
%

%


%
\subsection{Proof of Lemma \ref{lemma_subGaussian_norm}}
Since $\widetilde\varepsilon =\varepsilon\mathbf\Lambda^+$, $\Sigma_{\varepsilon} =\sigma^2\mathbf I_T$ and $\mathbf\Lambda\mathbf\Lambda^* = c(\tau,T)\mathbf I_{\tau}$,
\begin{eqnarray*}
 \Sigma_{\widetilde\varepsilon}^{-1/2} & = &
 ((\mathbf\Lambda^+)^*\Sigma_{\varepsilon}\mathbf\Lambda^+)^{-1/2}\\
 & = &
 \sigma^{-1}((\mathbf\Lambda^+)^*\mathbf\Lambda^+)^{-1/2}\\
 & = &
 \sigma^{-1}(\mathbf\Lambda\mathbf\Lambda^*)^{1/2}\\
 & = &
 \sigma^{-1}c(\tau,T)^{1/2}\mathbf I_{\tau}.
\end{eqnarray*}
Then, for any $x$ with $\|x\|=1$,
\begin{eqnarray*}
 \langle\widetilde\varepsilon_{1,.}\Sigma_{\widetilde\varepsilon}^{-1/2},x\rangle & = &
 \sigma^{-1}c(\tau,T)^{-1/2}\langle\varepsilon_{1,.}\mathbf\Lambda^*,x\rangle\\
 & = &
 c(\tau,T)^{-1/2}\|x\mathbf\Lambda\|
 \cdot\left\langle\varepsilon_{1,.}\Sigma_{\varepsilon}^{-1/2},\frac{x\mathbf\Lambda}{\|x\mathbf\Lambda\|}\right\rangle.
\end{eqnarray*}
Moreover,
\begin{displaymath}
\|x\mathbf\Lambda\|^2 =
x\mathbf\Lambda\mathbf\Lambda^*x^* = c(\tau,T)xx^* = c(\tau,T).
\end{displaymath}
Therefore,
\begin{eqnarray*}
 \mathfrak K_{\widetilde\varepsilon} & = &
 \sup_{\|x\|=1}
 \sup_{p\in [1,\infty[}
 p^{-1/2}\mathbb E(|\langle\widetilde\varepsilon_{1,.}\Sigma_{\widetilde\varepsilon}^{-1/2},x\rangle|^p)^{1/p}\\
 & = &
 c(\tau,T)^{-1/2}\\
 & &
 \times
 \sup_{\|x\|=1}\left\{
 \|x\mathbf\Lambda\|
 \sup_{p\in [1,\infty[}
 p^{-1/2}\mathbb E\left(\left|\left\langle\varepsilon_{1,.}\Sigma_{\varepsilon}^{-1/2},\frac{x\mathbf\Lambda}{\|x\mathbf\Lambda\|}\right\rangle\right|^p\right)^{1/p}\right\}
 =\mathfrak K_{\varepsilon}
\end{eqnarray*}
and finally, $\widetilde{\mathfrak K}_{\widetilde\varepsilon} =\widetilde{\mathfrak K}_{\varepsilon}$.
%


%
\subsection{Proof of Theorem \ref{model_selection}}
For short, let us denote
$$ \widehat{\mathbf M}_s := \widehat{\mathbf M}_{{\mathcal S}_{\widehat\tau_s,\widehat k_s}} .$$
Consider $\lambda\in ]0,1[$. On the one hand, by applying the Inequalities (\ref{preliminary_estimates_1}) and (\ref{preliminary_estimates_2}) of Lemma \ref{preliminary_estimates} successively:
\begin{eqnarray}
 R(\widehat{\mathbf M}_s)
 & \leqslant &
 \frac{r(\widehat{\mathbf M}_s) -\|\varepsilon\|_{F}^{2}}{1 -\lambda} +
 \frac{1}{\lambda(1 -\lambda)}
 \left\langle\varepsilon,\frac{\widehat{\mathbf M}_s
 -\mathbf M}{\|\widehat{\mathbf M}_s -\mathbf M\|_{F}}\right\rangle_{F}^{2}
 \nonumber\\
 & = &
 \frac{1}{1 -\lambda}\cdot\min_{(\tau,k)\in\mathcal T\times\mathcal K}\{r(\widehat{\mathbf M}_{\tau,k}) +
 \textrm{pen}_{s+\tau+k}(\tau,k) -\|\varepsilon\|_{F}^{2}\}
 \nonumber\\
 & &
 +\frac{1}{1 -\lambda}\left(-\textrm{pen}_{s+\tau+k}(\widehat\tau_s,\widehat k_s) +\frac{1}{\lambda}
 \left\langle\varepsilon,\frac{\widehat{\mathbf M}_s
 -\mathbf M}{\|\widehat{\mathbf M}_s -\mathbf M\|_{F}}\right\rangle_{F}^{2}\right)
 \nonumber\\
 \label{model_selection_1}
 & \leqslant &
 \frac{1}{1 -\lambda}\cdot
 \min_{(\tau,k)\in\mathcal T\times\mathcal K}
 \{(1 +\lambda)R(\widehat{\mathbf M}_{\tau,k}) +\textrm{pen}_{s+\tau+k}(\tau,k) +\psi_{\varepsilon}(\widehat{\mathbf M}_{\tau,k})\}\\
 & &
 +\frac{1}{1 -\lambda}(-\textrm{pen}_{s+\tau+k}(\widehat\tau_s,\widehat k_s) +\psi_{\varepsilon}(\widehat{\mathbf M}_s)),
 \nonumber
\end{eqnarray}
where
\begin{displaymath}
\psi_{\varepsilon}(\mathbf A) =
\frac{1}{\lambda}
\left\langle\varepsilon,\frac{\mathbf A
-\mathbf M}{\|\mathbf A -\mathbf M\|_{F}}\right\rangle_{F}^{2}
\textrm{ $;$ }
\forall\mathbf A\in\mathcal M_{d,T}(\mathbb R).
\end{displaymath}
On the other hand, consider $(\tau,k)\in\mathcal T\times\mathcal K$. Since $\widetilde\varepsilon =\varepsilon\mathbf\Lambda_{\tau}^{+}$,
\begin{eqnarray*}
 \psi_{\varepsilon}(\widehat{\mathbf M}_{\tau,k}) & = &
 \frac{1}{\lambda}
 \left\langle\varepsilon,\frac{(\widehat{\widetilde{\mathbf M}}_{\tau,k}
 -\widetilde{\mathbf M})\mathbf\Lambda_{\tau}}{\|\widehat{\mathbf M}_{\tau,k} -\mathbf M\|_{F}}\right\rangle_{F}^{2} =
 \frac{1}{\lambda}
 \left\langle\varepsilon\mathbf\Lambda_{\tau}^{+}\mathbf\Lambda_{\tau}\mathbf\Lambda_{\tau}^{*},\frac{\widehat{\widetilde{\mathbf M}}_{\tau,k}
 -\widetilde{\mathbf M}}{\|\widehat{\mathbf M}_{\tau,k} -\mathbf M\|_{F}}\right\rangle_{F}^{2}\\
 & \leqslant &
 \frac{1}{\lambda}\cdot
 \sup_{\mathbf A\in\Delta(\mathcal S_{\tau,k})^1}
 \langle\widetilde\varepsilon,\mathbf A\mathbf\Lambda_{\tau}^{*}\rangle_{F}^{2}
 \leqslant
 \frac{1}{\lambda}\|\mathbf\Lambda_{\tau}^{*}\|_{\textrm{op}}^{2}\cdot
 \textrm{rk}(\Delta(\mathcal S_{\tau,k})^1)\cdot
 \sigma_1(\widetilde\varepsilon)^2.
\end{eqnarray*}
As in the proof of Proposition \ref{concentration_inequality}, by Lemma \ref{concentration_inequality_lemma} and since
\begin{displaymath}
\|\Sigma_{\varepsilon\mathbf\Lambda_{\tau}^{+}}\|_{\textrm{op}}
\|\mathbf\Lambda_{\tau}^{*}\|_{\textrm{op}}^{2}
\leqslant\|\Sigma_{\varepsilon}\|_{\textrm{op}}
\|(\mathbf\Lambda_{\tau}\mathbf\Lambda_{\tau}^{*})^{-1}\|_{\normalfont{\textrm{op}}}\rho(\mathbf\Lambda_{\tau}\mathbf\Lambda_{\tau}^{*})
=\|\Sigma_\varepsilon\|_{\rm op},
\end{displaymath}
with probability larger than $1 - 2e^{-u}$,
\begin{displaymath}
\psi_{\varepsilon}(\widehat{\mathbf M}_{\tau,k})
\leqslant
\frac{2\mathfrak ck}{\lambda}(d +\tau + u)\widetilde{\mathfrak K}_{\varepsilon}\|\Sigma_\varepsilon\|_{\rm op}
 =\textrm{pen}_u(\tau,k).
\end{displaymath}
Take $u=s+\tau+k$, we obtain that with probability at least $1-2e^{-s-\tau-k}$,
\begin{displaymath}
\psi_{\varepsilon}(\widehat{\mathbf M}_{\tau,k})
\leqslant
\frac{2\mathfrak ck}{\lambda}(d +\tau + (s+2\tau+2k))
\widetilde{\mathfrak K}_{\varepsilon}
\|\Sigma_\varepsilon\|_{\rm op}
 =\textrm{pen}_{s+\tau+k}(\tau,k).
\end{displaymath}
Then, by a union bound,
\begin{eqnarray*}
 \mathbb P(\forall k,\forall \tau\text{: }\psi_{\varepsilon}(\widehat{\mathbf M}_{\mathcal{S}_{\tau,k}})\leqslant
 \textrm{pen}_{s+2\tau+2k}(\tau,k))
 & \geqslant &
 1 - 2 \sum_{(\tau,k)\in\mathcal{T}\times \mathcal{K}} e^{-s-\tau-k} \\
  & \geqslant &
 1 - 2 e^{-s}  \left(\sum_{\tau\geq 1} e^{-\tau} \right) \left(\sum_{k\geq 1} e^{-k} \right) 
 \\
   & \geqslant &
   1 - 2 e^{-s}.
\end{eqnarray*}
Together with Inequality (\ref{model_selection_1}), this gives, with probability at least $1-2e^{-s}$,
\begin{equation}
\label{steppp}
 R(\widehat{\mathbf M}_s)
 \leqslant  \frac{1}{1 -\lambda}\cdot
 \min_{(\tau,k)\in\mathcal T\times\mathcal K}
 \left\{(1 +\lambda)R(\widehat{\mathbf M}_{\tau,k}) +2 \textrm{pen}_{s+\tau+k}(\tau,k)  \right\}.
\end{equation}
Finally, follow the proof of Corollary~\ref{oracle_inequality_corollary_fini} to obtain, on the same event with probability at least $1-2e^{-s}$, for any $\tau$ and $k$,
$$
 R(\widehat{\mathbf M}_{\mathcal{S}_{\tau,k}})
   \leqslant 
  \frac{1 +\lambda }{1 -\lambda}\cdot
  \min_{\mathbf A\in\mathcal S_{\tau,k}}
  \left\| \mathbf{A}\mathbf{\Lambda}_\tau - \mathbf{M} \right\|_ F^2
  +\frac{2}{1 -\lambda}\textrm{pen}_{s+\tau+k}(\tau,k).
$$
Plugging this into~\eqref{steppp} gives, with probability at least $1-2e^{-s}$,
\begin{equation*}
 R(\widehat{\mathbf M}_s)
 \leqslant 
 \min_{(\tau,k)\in\mathcal T\times\mathcal K}
 \left\{\left(\frac{1+\lambda}{1-\lambda}\right)^2\min_{\mathbf A\in\mathcal S_{\tau,k}}
  \left\| \mathbf{A}\mathbf{\Lambda}_\tau - \mathbf{M} \right\|_ F^2 +\frac{4}{(1-\lambda)^2} \textrm{pen}_{s+\tau+k}(\tau,k)  \right\}.
\end{equation*}
Finally, note that $k\leqslant d$ so
\begin{align*}
\textrm{pen}_{s+\tau+k}(\tau,k)
& = \frac{2\mathfrak ck}{\lambda}(d +\tau + (s+\tau+k))\widetilde{\mathfrak K}_{\varepsilon}\|\Sigma_\varepsilon\|_{\rm op}
\\
& \leqslant \frac{4\mathfrak ck}{\lambda}(d +\tau + s)\widetilde{\mathfrak K}_{\varepsilon}\|\Sigma_\varepsilon\|_{\rm op}
\end{align*}
and so, with probability at least $1-2e^{-s}$,
\begin{small}
\begin{displaymath}
R(\widehat{\mathbf M}_s)
\leqslant
\min_{(\tau,k)\in\mathcal T\times\mathcal K}
\left\{\left(\frac{1+\lambda}{1-\lambda}\right)^2\min_{\mathbf A\in\mathcal S_{\tau,k}}
\left\| \mathbf{A}\mathbf{\Lambda}_\tau - \mathbf{M} \right\|_ F^2 +\frac{16\mathfrak ck}{\lambda (1 -\lambda)^2}(d+\tau+s)
\widetilde{\mathfrak K}_{\varepsilon}\|\Sigma_{\varepsilon}\|_{\rm op}\right\}.
\end{displaymath}
\end{small}
\newline
This ends the proof.
\end{document}